\documentclass[11pt]{article}
\usepackage{amsmath,amssymb,amsthm}
\usepackage{graphicx,subfigure,float,url}
\usepackage{pdfsync}
\usepackage[english]{babel}
\usepackage[utf8]{inputenc}
\usepackage{tipa}
\usepackage{tikz}

\newlength\imagewidth
\newlength\imagescale

\usepackage{bbold}
\usepackage{bbm}

\usepackage{fancybox}
\usepackage{listings}
\usepackage{mathtools}
\usepackage{amsfonts}
\usepackage[bottom]{footmisc}
\usepackage{verbatim}

\usepackage{colortbl}

\usepackage{float}
\usepackage{makeidx}
\newcommand{\deriv}{\frac{d}{dt}}

\normalbaselines

\newtheorem{Theorem}{Theorem}[section]

\newtheorem{Definition}[Theorem]{Definition}

\newtheorem{Lemma}[Theorem]{Lemma}

\newtheorem{Corollary}[Theorem]{Corollary}

\newtheorem{Remark}[Theorem]{Remark}

\newcommand{\RR}{{{\rm I} \kern -.15em {\rm R} }}

\newcommand{\C}{{{\rm l} \kern -.50em {\rm C} }}

\newcommand{\nat}{{{\rm I} \kern -.15em {\rm N} }}

\newcommand{\tr}{\tilde{R}}

\newcommand{\be}{\begin{equation}}
\newcommand{\ee}{\end{equation}}
\newcommand{\beq}{\begin{eqnarray}}
\newcommand{\eeq}{\end{eqnarray}}
\newcommand{\beqs}{\begin{eqnarray*}}
\newcommand{\eeqs}{\end{eqnarray*}}
\newcommand{\bt}{\begin{Theorem}}
\newcommand{\et}{\end{Theorem}}
\newcommand{\br}{\begin{Remark}}
\newcommand{\er}{\end{Remark}}
\newcommand{\bc}{\begin{Corollary}}
\newcommand{\ec}{\end{Corollary}}
\newcommand{\bl}{\begin{Lemma}}
\newcommand{\el}{\end{Lemma}}
\newcommand{\bd}{\begin{definition}}
\newcommand{\ed}{\end{definition}}

\renewcommand{\geq}{\geqslant}
\renewcommand{\leq}{\leqslant}

\graphicspath{{figures/}}

\title{On the control of the Hegselmann-Krause model \\ with leadership and time delay}
\author{
\and
Alessandro Paolucci\footnote{Dipartimento di Ingegneria e Scienze dell'Informazione e Matematica, Universit\`{a} di L'Aquila, Via Vetoio, Loc. Coppito, 67010 L'Aquila Italy (\texttt{alessandro.paolucci2@graduate.univaq.it}).}
\and
Cristina Pignotti\footnote{Dipartimento di Ingegneria e Scienze dell'Informazione e Matematica, Universit\`{a} di L'Aquila, Via Vetoio, Loc. Coppito, 67010 L'Aquila Italy (\texttt{cristina.pignotti@univaq.it}).}
}

\date{}

\begin{document}

\textwidth=160 mm

\textheight=225mm

\parindent=8mm

\frenchspacing

\maketitle

\begin{abstract}
We analyze Hegselmann-Krause opinion formation models with leadership in presence of time delay effects. In particular, we consider a model with pointwise time variable time delay and a model with a distributed delay. In both cases we show that, when  the delays satisfy suitable smallness conditions, then the leader can control the system, leading the group to any prefixed state. Some numerical tests illustrate our analytical results.
\end{abstract}

\vspace{5 mm}

\def\qed{\hbox{\hskip 6pt\vrule width6pt
height7pt
depth1pt  \hskip1pt}\bigskip}

\section{Introduction}
In recent years, multi-agent models have attracted the interest of several researchers in different scientific areas, such as biology, economics, robotics, control theory, social sciences (see e.g. \cite{Bullo, Cama, CCR, CCP, Carrillo, CS1, Desai, Jad, Marsan, MT, Sole, Stro, XWX}). An important feature analyzed is the possible emergence of self-organization leading to globally collective behaviours.

Specifically, it is important to investigate as a group of agents can reach a consensus starting from different opinions or, for instance, as the opinion of a leader can influence other individuals. We can observe more and more as in the real life social networks can contribute to condition the behavior of the individuals and their opinions. For this reason, many political groups and statesmen try to diffuse their opinions and ideas through social media. Therefore, for the importance of applications in social sciences, among multi-agent systems a key role is played by the so-called opinion formation models.

In this paper, we will focus on a Hegselmann-Krause type opinion formation model. The idea is that each opinion changes according with the opinions of other agents taking place in the bounded domain of confidence. The original model was proposed in  \cite{HK} (see also \cite{BHT}). Then, a number of variants and generalizations have been introduced (see e.g. \cite{Aydogdu, Bellomo, BN, CFT, CFT2, Castellano, CF, During, JM, Piccoli, Campi}).
In particular, in \cite{WCB}, a control problem has been studied for the Hegselmann-Krause model with leadership. Namely, through suitable controls, the leader influences other agents so that the whole system tends to any desired consensus state. See also \cite{D} for analogous result in the one-dimensional case.

A significant aspect that one has to consider dealing with problems from applied sciences is the natural presence of time delays.
Time delay effects frequently appear in applications and physical or biological models and it is well-known that they can induce instability phenomena. In particular, in multi-agent models, the presence of time delays makes the systems more difficult to deal with since they destroy some typical  symmetries and geometrical features.  Multi-agent models with time delay have been proposed and analyzed in recent papers dealing, e.g., with Cucker--Smale models (\cite{LW, CH, CL, CP, EHS, HM, PR, PT}), thermomechanical Cucker-Smale systems (\cite{DH}) and opinion formation models (\cite{Lu, CPP, P, H}). In these papers, under different sets of assumptions, convergence to consensus results are proved, at least for sufficiently small time delays.
As far as we know, control problems for the Hegselmann--Krause opinion formation model with delays have not yet been studied so far. So, here, our aim is to show that, in presence of a leader, a Hegselmann--Krause model with time delays  can be controlled  in order to reach any consensus state, namely any target opinion. Our results require that the time delays satisfy suitable smallness assumptions. As a concrete application one can think to the case of social networks, through which the leader can inform and condition other agents in very short times, namely with a small time delay.

 Let $N\in\nat$ be the number of agents and let $x_i:[0,+\infty)\to \RR^d$ be the opinion of the $i$-th agent, for any $i=1,\dots,N$. Moreover, let us consider the leader's opinion $x_0:[0,+\infty)\to \RR^d$. Then, the dynamics is governed by the following modified Hegselmann-Krause model with leadership:
\begin{equation}
\label{main_model}
\begin{array}{l}
\displaystyle{ \frac{dx_0}{dt} (t)
=u(t),}\\\medskip
\displaystyle{ \frac{dx_i}{dt}(t)=\frac{1}{N} \sum_{j\neq i} a_{ij}(t) (x_j(t-\tau(t))-x_i(t))}\\
\hspace{3
 cm}
\displaystyle{+\gamma \phi (|x_0(t-\tau(t))-x_i(t)|)(x_0(t-\tau(t))-x_i(t)),}
\end{array}
\end{equation}
for all $t\geq 0$, with continuous initial data, for $i=1, \dots, N,$
\begin{equation}
\label{initial_condition}
\displaystyle{ x_0(s)=x_{0,0}(s)}, \quad
\displaystyle{ x_i(s)=x_{i,0}(s)}, \quad \forall\ s\in[-\bar\tau,0].
\end{equation}
Here, the control $u:[0,+\infty)\to \RR^d$ is a measurable function, which represents the result of the leader's opinion and $\tau(\cdot):[0,+\infty)\to[0,+\infty)$ is the time-dependent time delay.
We assume that $\tau (\cdot)$ is continuous and bounded, i.e.
$$
\tau(t)\leq \bar{\tau}, \quad t\geq 0,
$$
where $\bar\tau>0$ is the constant in \eqref{initial_condition}. The influence function is of the form $$a_{ij}(t)=a(|x_j(t-\tau(t))-x_i(t)|),$$
where $a:[0,+\infty)\to [0,+\infty)$ is a continuous non-increasing cut-off function. In particular, we assume:
$$
a(s)=1 \quad \mbox{\rm for}\ s\in [0,\delta], \quad a(s)=0 \quad \mbox{\rm for}\ s\ge r,$$
where $\delta, r \in \RR^+$, with $r>\delta.$ Moreover, $\gamma$ is a positive parameter representing  the strength of the leader's opinion and $\phi(\cdot):[0,+\infty)\to \RR^+$ is the leader's influence function on the agents. We assume $\phi$ non-increasing and  Lipschitz continuous, with Lipschitz constant $L>0$. Without loss of generality, we can take $\phi(0)=1$.

Furthermore, we consider a second system with distributed time-varying time delay. In this case, each agent is influenced by other agents' opinions along an entire time interval $[t-\tau(t),t]$. This generalizes the dynamics of the model \eqref{main_model}, in which agents are influenced only according to the opinions at a previous time $t-\tau(t)$, for any $t\geq 0$. In particular, our model is the following:
\begin{equation}\label{ds}
\begin{array}{l}
\displaystyle{\frac {d x_0}{dt}(t)=u(t),}\\\medskip
\displaystyle{\frac{d x_i}{dt}(t)=\frac{1}{Nh(t)}\sum_{j\neq i}\int_{t-\tau(t)}^t \beta(t-s) a_{ij}(t;s) (x_j(s)-x_i(t))ds}\\
\hspace{3 cm}
\displaystyle{+\frac{\gamma}{h(t)} \int_{t-\tau(t)}^t \beta(t-s)\phi(|x_0(s)-x_i(t)|)(x_0(s)-x_i(t))ds,}
\end{array}
\end{equation}
with continuous initial data, for $i=1, \dots, N,$
\begin{equation}\label{initial_data_dd}
x_i(s)=x_{i,0}(s),\quad
x_0(s)=x_{0,0}(s), \quad \forall\  s\in[-\bar\tau, 0].
\end{equation}
Here, $a_{ij}(t;s):=a(|x_j(s)-x_i(t)|)$ and $\phi$ satisfy previous assumptions. Moreover, $\tau(\cdot):[0,+\infty)\to \RR^+$ is the time-dependent time delay, which we suppose continuous and  bounded. In particular, we assume that there exist two constants $\bar\tau, \tau_*>0$ such that $$\tau_*\leq \tau(t)\leq \bar\tau$$ for any $t\geq 0$. Furthermore, $\beta(\cdot):[0,\bar\tau]\to\RR^+$ is a weight function satisfying
\begin{equation}\label{htau}
\int_0^{\tau_*} \beta(s)ds> 0.
\end{equation}
Finally, $h(\cdot)$ is defined as
$$
h(t):=\int_0^{\tau(t)} \beta(s)ds,
$$
for any $t\geq 0.$ Then, it is always strictly positive by assumption \eqref{htau}.

We notice that if $\beta(s)=\delta_{\tau(t)}(s)$, then system \eqref{ds} can be rewritten as in \eqref{main_model}.

Distributed time-delay systems have been analyzed very recently in \cite{CPP, P, LLW}, in which the authors improved and generalized the theory of delayed Cucker-Smale and Hegselmann-Krause models with pointwise delay.

A well-posedness result for systems \eqref{main_model}-\eqref{initial_condition} or \eqref{ds}-\eqref{initial_data_dd} with continuous initial data can be found in \cite{HL} (see also \cite{Halanay}). Here, we are interested in the  consensus behavior of the solutions subject to leader's control strategies.

\begin{Definition}
We say that a measurable control $t\mapsto u(t)$ is \textbf{admissible} if there exists a constant $M>0$ such that $||u||_{L^\infty} \leq M$.
\end{Definition}

\begin{Definition}
Let $\{x_i\}_{i=0,1,\dots,N}$ be a solution to either \eqref{main_model}-\eqref{initial_condition} or \eqref{ds}-\eqref{initial_data_dd}. We say that the solution converges to \textbf{consensus} if and only if
$$
\lim_{t\to +\infty} |x_i(t)-x_j(t)|=0,
$$
for any $i,j\in\{0,1,\dots, N\}.$
\end{Definition}
The rest of the paper is organized as follows. In Section \ref{section_consensus} we show the existence of a control function $u$ which modifies the leader's opinion so that consensus for system \eqref{main_model} is achieved. In Section \ref{section_2}, we prove the local controllability property, namely under a smallness assumption on the time delay and the parameters involved in the dynamics we will be able to find a control which steers all agents to a given state. Sections \ref{sec_distr} and \ref{sec5} are devoted to consensus and controllability properties for solutions to system \eqref{ds}. Finally, in Section \ref{numeric} some numerical results are shown to illustrate the theory of previous sections.
\section{Consensus behavior}\label{section_consensus}
In this section, inspired by \cite{WCB}, we prove that there exists an admissible control $u$ so that consensus among all agents (leader included) occurs. Let us define the distance from the leader's opinion
\begin{equation}\label{dzero}
d_0(t):= \max_{i\in \{1,\dots,N\}} |x_i(t)-x_0(t)|,\quad t\ge -\bar\tau ,
\end{equation}
and let $p=p(t)\in \{1,...,N\}$ be such that
$$
|x_p(t)-x_0(t)|=d_0(t).
$$
Without loss of generality, we can take $p=1$. Furthermore, consider the
control
\begin{equation}
\label{control}
u(t)=\gamma \alpha(t)\sum_{j=1}^N \phi(|x_j(t-\tau(t))-x_0(t)|)(x_j(t-\tau(t))-x_0(t)),
\end{equation}
where
$$
\alpha(t):= \frac{1}{2}\min \left\{ \frac{\phi(|x_1(t-\tau(t))-x_0(t)|)}{N}, \frac{2M}{\gamma \sum_{j= 1}^N |x_j(t-\tau(t))-x_0(t)|}   \right\} .
$$
By definition, $u$ is an admissible control. Moreover, we have the following lemma.
\begin{Lemma}\label{lemma1}
Let $\{x_i\}_{i=0,1,\dots,N}$ be the solution to \eqref{main_model}-\eqref{initial_condition}, with control as in \eqref{control}. Let $x_{0,0}$, $x_{i,0}:[-\bar\tau,0]\to \RR^d$ be continuous initial data and let
$$
R:=\max_{s\in[-\bar\tau,0]}\max_{i=0,1,\dots,N} |x_i(s)|.
$$
Then,
\begin{equation}\label{tesi_lemma1}
\max_{i\in\{ 0,1,...,N\}} |x_i(t)|\leq R,
\end{equation}
for any $i\in\{1,\dots,N\}$ and for any $t\geq 0$.
\end{Lemma}
\begin{proof}
Let $\epsilon>0$. Define $R_\epsilon:=R+\epsilon$ and set
$$
S_\epsilon :=\{ t>0\ : \ \max_{i\in\{ 0,1,...,N\}} |x_i(s)|<R_\epsilon, \ \forall\ s\in[0,t)\}.
$$
By continuity, $S_\epsilon\neq \emptyset$. Let $T_\epsilon:=\sup S_\epsilon$ and suppose by contradiction that $T_\epsilon<+\infty$. Then, we have that
\begin{equation}\label{contr1}
\lim_{t\to T_\epsilon -} \max_{i\in\{ 0,1,...,N\}} |x_i(t)|=R_\epsilon.
\end{equation}
On the other hand, we have that for any $i\in\{1,...,N\}$ and for any $t<T_\epsilon$,
$$
\begin{array}{l}
\displaystyle{\frac 1 2 \frac{d}{dt}|x_i(t)|^2=\left\langle x_i(t),\frac{dx_i}{dt}(t) \right\rangle}\\
\hspace{3 cm}
\displaystyle{=\frac 1 N \sum_{j\neq i} a_{ij}(t) \langle x_j(t-\tau(t))-x_i(t),x_i(t)\rangle}\\
\hspace{5 cm}
\displaystyle{+\gamma \phi (|x_0(t-\tau(t))-x_i(t)|)\langle x_0(t-\tau(t))-x_i(t),x_i(t)\rangle}\\
\hspace{3 cm}
\displaystyle{= \frac 1 N \sum_{j\neq i}a_{ij}(t) \left( \langle x_j(t-\tau(t)),x_i(t)\rangle -|x_i(t)|^2\right)}\\
\hspace{5 cm}
\displaystyle{+\gamma \phi(|x_0(t-\tau(t))-x_i(t)|)\left( \langle x_0(t-\tau(t)),x_i(t)\rangle-|x_i(t)|^2\right) .}
\end{array}
$$
By using Cauchy-Schwarz inequality and the fact that $t-\tau(t)\in S_\epsilon$ yields
$$
\begin{array}{l}
\displaystyle{\frac 1 2 \frac{d}{dt}|x_i(t)|^2\leq \frac 1 N \sum_{j\neq i} a_{ij}(t) |x_i(t)|\left( R_\epsilon-|x_i(t)|\right)+\gamma \phi(|x_0(t-\tau(t))-x_i(t)|)|x_i(t)|\left( R_\epsilon-|x_i(t)|\right).}
\end{array}
$$
Since $|x_i(t)|<R_\epsilon$, then
$$
 \frac{d}{dt}|x_i(t)| \leq (\gamma+1)(R_\epsilon-|x_i(t)|),
$$
and by using Gronwall estimate we get
\begin{equation}\label{eq_per_contr1}
|x_i(t)|\leq e^{-(\gamma+1)t} \left( |x_i(0)|-R_\epsilon\right)+R_\epsilon <R_\epsilon, \quad\forall\ t<T_\epsilon.
\end{equation}
Moreover, if we consider the leader's opinion with control $u$ defined in \eqref{control}, we have that
$$
\begin{array}{l}
\displaystyle{ \frac 1 2 \frac{d}{dt}|x_0(t)|^2 =\gamma \alpha(t)\sum_{j=1}^N \phi(|x_j(t-\tau(t))-x_0(t)|)\left( \langle x_j(t-\tau(t)),x_0(t)\rangle -|x_0(t)|^2\right) .}
\end{array}
$$
By Cauchy-Schwarz inequality and using the fact that for any $j\in\{1,...,N\}$ and $t\geq 0$, $|x_j(t-\tau)|\leq R_\epsilon$, we obtain
$$
\begin{array}{l}
\displaystyle{ \frac 1 2 \frac{d}{dt}|x_0(t)|^2  \leq \gamma \alpha(t)\sum_{j=1}^N \phi(|x_j(t-\tau(t))-x_0(t)|)|x_0(t)|\left( R_\epsilon-|x_0(t)|\right)}\\
\hspace{2 cm}
\displaystyle{ \leq \gamma |x_0(t)|\left( R_\epsilon -|x_0(t)|\right),}
\end{array}
$$
where we have used the fact that for any $t\leq T_\epsilon$, $|x_0(t)|\leq R_\epsilon$ and $N\alpha(t)\leq 1$. Therefore, using once again Gronwall estimate yields
\begin{equation}\label{eq_per_contr2}
\begin{array}{l}
\displaystyle{|x_0(t)|\leq e^{-\gamma t}\left( |x_0(0)|-R_\epsilon\right)+R_\epsilon<R_\epsilon.}
\end{array}
\end{equation}
Therefore, by inequalities \eqref{eq_per_contr1}-\eqref{eq_per_contr2} we can conclude that
$$
\lim_{t\to T_\epsilon -} \max_{i\in\{ 0,1,...,N\}} |x_i(t)|< R_\epsilon.
$$
This is in contradiction with \eqref{contr1}. Then, $T_\epsilon=+\infty$. By sending $\epsilon\to 0$, we obtain the thesis.
\end{proof}
Lemma \ref{lemma1} guarantees not only global existence of solution to \eqref{main_model} with $u$ as in \eqref{control}, but also a control from below of the influence function $\phi$. Indeed, in this case, since $|x_0(t-\tau(t))-x_i(t)|\leq 2R$ and $\phi$ is non-increasing, we have that
\begin{equation}\label{below}
\phi(|x_0(t-\tau(t))-x_i(t)|)\geq \phi(2R),
\end{equation}
for any $t\geq 0$ and for any $i\in\{1,...,N\}$.

\begin{Remark}\label{Rstar}{\rm
Since the system is invariant by translation one could improve \eqref{below} by considering the relative distance among the opinions. Indeed, arguing as in Lemma 2.2 of \cite{CPP}, we can deduce
$$
\phi(|x_0(t-\tau(t))-x_i(t)|)\geq \phi(2R^*),$$
with
$$R^*:=\min_{C\in\RR^d} \,\max_{[-\tau(0),0]}\,\max_{1\le i\le N}\vert x_i(t)-C\vert\,.$$
Then, all the following results hold true also with  $R^*$ in place of $R.$  }
\end{Remark}

Note that the function $d_0(t)$ defined in \eqref{dzero} may be not differentiable for some $t>0$. Therefore, we will use the upper Dini derivative. The upper Dini derivative of a continuous function $F(\cdot)$ is defined as
$$
D^+F(t):=\limsup_{h\to 0+} \frac{F(t+h)-F(t)}{h}.
$$
Before proving the consensus result, we need the following lemma.
\begin{Lemma}\label{lemma2}
Let $\{x_i\}_{i=0,1,...,N}$ be the solution to \eqref{main_model}-\eqref{initial_condition}, with control as in \eqref{control}. Define
$$
\sigma_\tau(t):=\int_{t-\bar\tau}^t \left[ \max_{j\in\{1,...,N\}}\left\vert \frac{d{x}_j}{ds}(s)\right\vert +\vert u(s)\vert \right] ds, \quad \forall\ t\ge \bar\tau .
$$
Then, for any $t\geq \bar{\tau},$
\begin{equation}\label{tesi_lemma2}
D^+ d_0(t) \leq -\frac \gamma 2 \phi(2R) d_0(t)+R_\gamma \sigma_\tau(t),
\end{equation}
where $R_\gamma:=\gamma LR+\gamma+1$.
\end{Lemma}
\begin{proof} As before, without loss of generality, we assume that
$d_0(t)= \vert x_1(t)-x_0(t)\vert.$
From \eqref{main_model} and \eqref{control} we have that, for any $t\geq 0,$
$$
\begin{array}{l}
\displaystyle{\frac 1 2 D^+|x_1(t)-x_0(t)|^2\leq \frac 1 N \sum_{j\neq 1} a_{1j}(t)\langle x_j(t-\tau(t))-x_1(t),x_1(t)-x_0(t)\rangle }\\
\hspace{3.75 cm}
\displaystyle{+\gamma \phi(|x_0(t-\tau(t))-x_1(t)|)\langle x_0(t-\tau(t))-x_1(t),x_1(t)-x_0(t)\rangle }\\
\hspace{3.75 cm}
\displaystyle{-\gamma \alpha(t)\sum_{j=1}^N \phi(|x_j(t-\tau(t))-x_0(t)|)\langle x_j(t-\tau(t))-x_0(t),x_1(t)-x_0(t)\rangle}\\
\hspace{3.5 cm}
\displaystyle{=\frac 1 N \sum_{j\neq 1} a_{1j}(t)\langle x_j(t-\tau(t))-x_j(t),x_1(t)-x_0(t)\rangle}\\
\hspace{3.75 cm}
\displaystyle{ +\frac 1 N \sum_{j=1}^N a_{1j}(t) \langle x_j(t)-x_1(t),x_1(t)-x_0(t)\rangle}\\
\hspace{3.75 cm}
\displaystyle{+\gamma \phi(|x_0(t-\tau(t))-x_1(t)|)\langle x_0(t-\tau(t))-x_0(t),x_1(t)-x_0(t)\rangle}\\
\hspace{3.75 cm}
\displaystyle{ -\gamma\phi(|x_0(t-\tau(t))-x_1(t)|)|x_1(t)-x_0(t)|^2}\\
\hspace{3.75 cm}
\displaystyle{ -\gamma \alpha(t)\sum_{j=1}^N \phi(|x_j(t-\tau(t))-x_0(t)|)\langle x_j(t-\tau(t))-x_j(t),x_1(t)-x_0(t)\rangle}\\
\hspace{3.75 cm}
\displaystyle{-\gamma\alpha(t)\sum_{j=1}^N \phi(|x_j(t-\tau(t))-x_0(t)|)\langle x_j(t)-x_0(t),x_1(t)-x_0(t)\rangle.}
\end{array}
$$
Once again, we can write for any $j\in\{1,...,N\}$
$$
\langle x_j(t)-x_1(t),x_1(t)-x_0(t)\rangle=\langle x_j(t)-x_0(t),x_1(t)-x_0(t)\rangle -|x_1(t)-x_0(t)|^2.
$$
Moreover, for any $t\geq \bar\tau$ and for any $j\in\{1,...,N\},$
$$
|x_j(t-\tau(t))-x_j(t)|\leq \int_{t-\tau(t)}^t \left|\frac{d{x}_j}{ds}(s)\right|ds\leq \int_{t-\bar\tau}^t \max_{j\in\{1,...,N\}}\left|\frac{d{x}_j}{ds}(s)\right|ds,
$$
where we have used the positivity of the integrand function and the fact that $\tau(t)\leq \bar\tau$ for any $t\geq 0$. Analogously,
$$
|x_0(t-\tau(t))-x_0(t)|\leq \int_{t-\bar\tau}^t \left|\frac{d{x}_0}{ds}(s)\right|ds.
$$
Then, we obtain
$$
\begin{array}{l}
\displaystyle{ \frac 1 2 D^+ |x_1(t)-x_0(t)|^2\leq \frac 1 N \sum_{j=1}^N a_{1j}(t)\langle x_j(t)-x_0(t), x_1(t)-x_0(t)\rangle}\\
\hspace{3.75 cm}
\displaystyle{ +|x_1(t)-x_0(t)|\int_{t-\bar\tau}^t \max_{j\in\{1,...,N\}} \left|\frac{d{x}_j}{ds}(s)\right| ds-\frac 1 N \sum_{j=1}^N a_{1j}(t) |x_1(t)-x_0(t)|^2}\\
\hspace{3.75 cm}
\displaystyle{-\gamma \phi(|x_0(t-\tau(t))-x_1(t)|)|x_1(t)-x_0(t)|^2}\\
\hspace{3.75 cm}
\displaystyle{+\gamma |x_1(t)-x_0(t)| \int_{t-\bar\tau}^t \left|\frac{d{x}_0}{ds}(s)\right|ds}\\
\hspace{3.75 cm}
\displaystyle{+\gamma |x_1(t)-x_0(t)|\int_{t-\bar\tau}^t \max_{j\in\{1,...,N\}} \left|\frac{d{x}_j}{ds}(s)\right|ds}\\
\hspace{3.75 cm}
\displaystyle{-\gamma\alpha(t)\sum_{j=1}^N \phi(|x_j(t-\tau(t))-x_0(t)|)\langle x_j(t)-x_0(t),x_1(t)-x_0(t)\rangle.}
\end{array}
$$
Therefore, putting together the first term and the last one, and noticing that by definiton of $|x_1(t)-x_0(t)|$
$$
\langle x_j(t)-x_0(t),x_1(t)-x_0(t)\rangle \leq |x_1(t)-x_0(t)|^2,
$$
we get, for any $t\geq \bar\tau$, 
$$
\begin{array}{l}
\displaystyle{\frac 1 2 D^+|x_1(t)-x_0(t)|^2\leq \sum_{j=1}^N  \left( \left| \frac 1 N a_{1j}(t)-\gamma \alpha(t)\phi(|x_j(t-\tau(t))-x_0(t)|)\right| -\frac 1 N a_{1j}(t)\right) |x_1(t)-x_0(t)|^2}\\
\hspace{3.75 cm}
\displaystyle{ -\gamma \phi(|x_0(t-\tau(t))-x_1(t)|)|x_1(t)-x_0(t)|^2+(\gamma+1) |x_1(t)-x_0(t)|\sigma_\tau(t).}\\
\end{array}
$$
Now, let us define for any $j\in\{ 1,...,N\}$,
$$
b_j:= \left| \frac 1 N a_{1j}(t)-\gamma \alpha(t)\phi(|x_j(t-\tau(t))-x_0(t)|)\right| -\frac 1 N a_{1j}(t).
$$
If $j\in\{1,...,N\}$ is such that
$$
\frac 1 N a_{1j}(t)-\gamma \alpha(t)\phi(|x_j(t-\tau(t))-x_0(t)|)\geq 0,
$$
then, $b_j<0$. Otherwise, we have that
$$
b_j\leq \gamma\alpha(t)\phi(|x_j(t-\tau(t))-x_0(t)|),
$$
which gives us
$$
\begin{array}{l}
\displaystyle{\frac 1 2 D^+ |x_1(t)-x_0(t)|^2}\\
\hspace{2 cm}\displaystyle{
\leq \Big( \gamma \alpha(t)\sum_{j=1}^N \phi(|x_j(t-\tau(t))-x_0(t)|)-\gamma \phi(|x_0(t-\tau(t))-x_1(t)|)\Big) |x_1(t)-x_0(t)|^2}\\
\hspace{3.75 cm}
\displaystyle{+(\gamma+1)\sigma_\tau(t)|x_1(t)-x_0(t)|.}
\end{array}
$$
By definition of $\alpha(\cdot)$ and \eqref{below}, we get
$$
\begin{array}{l}
\displaystyle{\frac 1 2 D^+|x_1(t)-x_0(t)|^2\leq \frac \gamma 2 \left| \phi(|x_1(t-\tau(t))-x_0(t)|)-\phi(|x_0(t-\tau(t))-x_1(t)|)\right| |x_1(t)-x_0(t)|^2}\\
\hspace{3.75 cm}
\displaystyle{-\frac \gamma 2 \phi(2R) |x_1(t)-x_0(t)|^2+(\gamma+1)\sigma_\tau(t)|x_1(t)-x_0(t)|.}
\end{array}
$$
Since $\phi(\cdot)$ is Lipschitz continuous, we obtain
$$
\begin{array}{l}
\displaystyle{ \frac 1 2 D^+|x_1(t)-x_0(t)|^2\leq \frac \gamma 2 L |x_1(t)-x_0(t)|^2 \sigma_\tau(t)-\frac \gamma 2 \phi(2R) |x_1(t)-x_0(t)|^2}\\
\hspace{3.75 cm}
\displaystyle{+(\gamma+1)\sigma_\tau(t)|x_1(t)-x_0(t)|,}
\end{array}
$$
and using once again \eqref{tesi_lemma1} yields
$$
\begin{array}{l}
\displaystyle{ \frac 1 2 D^+|x_1(t)-x_0(t)|^2\leq -\frac \gamma 2 \phi(2R)|x_1(t)-x_0(t)|^2+R_\gamma \sigma_\tau(t)|x_1(t)-x_0(t)|,}
\end{array}
$$
from which we obtain inequality \eqref{tesi_lemma2}.
\end{proof}
\begin{Lemma}\label{lemma3}
Let $\{ x_i\}_{i\in\{0,1,...,N\}}$ be the solution to \eqref{main_model}-\eqref{initial_condition}, with control as in \eqref{control}. Then, for any $t\ge\bar\tau,$ we have 
\begin{equation}\label{tesi1_lemma3}
\left|\frac{d{x}_0}{dt}(t)\right|\leq \gamma d_0(t)+\gamma\sigma_\tau(t)
\end{equation}
and
\begin{equation}\label{tesi2_lemma3}
\max_{i\in\{1,...,N\}} \left|\frac{d {x}_i }{dt}(t)\right|\leq (2+\gamma)d_0(t)+(1+\gamma)\sigma_\tau(t).
\end{equation}
\end{Lemma}
\begin{proof}
By direct calculation, we have that, for any $t\geq \bar\tau$,
$$
\begin{array}{l}
\displaystyle{\left|\frac{d {x}_0}{dt}(t)\right|=|u(t)|\leq \gamma \alpha(t)\sum_{j=1}^N \phi(|x_j(t-\tau(t))-x_0(t)|)|x_j(t-\tau(t))-x_j(t)|}\\
\hspace{4 cm}
\displaystyle{ +\gamma\alpha(t)\sum_{j=1}^N \phi(|x_j(t-\tau(t))-x_0(t)|)|x_j(t)-x_0(t)|}\\
\hspace{2.85 cm}
\displaystyle{\leq \gamma N \alpha(t)\sigma_\tau(t)+\gamma N \alpha (t)|x_1(t)-x_0(t)|.}
\end{array}
$$
Noticing once again that $N\alpha(t)\leq 1$ for any $t\geq 0$, we obtain \eqref{tesi1_lemma3}. In order to prove inequality \eqref{tesi2_lemma3}, for any $i\in\{1,...,N\}$ we have that
$$
\begin{array}{l}
\displaystyle{ \left|\frac{d{x}_i}{dt}(t)\right|\leq \frac 1 N \sum_{j\neq i} a_{ij}(t)|x_j(t-\tau(t))-x_j(t)|+\frac 1 N \sum_{j\neq i} a_{ij}(t)|x_j(t)-x_0(t)|}\\
\hspace{1.25 cm}
\displaystyle{+\frac 1 N \sum_{j\neq i} a_{ij}(t)|x_0(t)-x_i(t)|+\gamma\phi(|x_0(t-\tau(t))-x_i(t)|)|x_0(t-\tau(t))-x_0(t)|}\\
\hspace{1.25 cm}
\displaystyle{+\gamma\phi(|x_0(t-\tau(t))-x_i(t)|)|x_0(t)-x_i(t)|}\\
\hspace{1.25 cm}
\displaystyle{\leq (2+\gamma)|x_1(t)-x_0(t)|+(1+\gamma)\sigma_\tau(t), \quad \forall \ t\ge\bar{\tau}.}
\end{array}
$$
Hence, by taking the maximum for $i\in\{1,...,N\}$ we obtain \eqref{tesi2_lemma3}.
\end{proof}
We are now ready to prove the main theorem of this section, namely the consensus behavior of the solution to \eqref{main_model}-\eqref{initial_condition}.
\begin{Theorem}
Let $\{x_i\}_{i\in\{0,1,...,N\}}$ be the solution to \eqref{main_model}-\eqref{initial_condition}, with control as in \eqref{control}. If
\begin{equation}\label{cond_tau}
\bar\tau< \ln\left( 1+\frac{\gamma\phi(2R)}{\gamma(1+2\gamma)\phi(2R)+4(1+\gamma)R_\gamma}\right),
\end{equation}
then there exist $\beta^*,C>0$ such that, for any $t\geq 0,$
\begin{equation}\label{tesi_teo1}
d_0(t)\leq Ce^{-\beta^*t}.
\end{equation}
\end{Theorem}
\begin{proof}
Let us define the following Lyapunov functional
$$
L(t):= d_0(t)+\beta\int_{t-\bar\tau}^t e^{-(t-s)} \int_s^t \left( \max_{j\in\{1,...,N\}} \left|\frac{d{x}_j}{d\sigma}(\sigma)\right|+\left|\frac{d{x}_0}{d\sigma}(\sigma)\right|\right) d\sigma ds,\quad t\ge\bar\tau,
$$
where $\beta>0$ is a suitable parameter which will be determined later. Then, differentiating we obtain
$$
\begin{array}{l}
\displaystyle{D^+L(t)= D^+d_0(t)-\beta e^{-\bar\tau}\sigma_\tau(t)}\\
\hspace{2 cm}
\displaystyle{-\beta\int_{t-\bar\tau}^t e^{-(t-s)}\int_s^t \left( \max_{j\in\{1,...,N\}} \left|\frac{d{x}_j}{d\sigma}(\sigma)\right|+\left|\frac{d{x}_0}{d\sigma}(\sigma)\right|\right) d\sigma ds}\\
\hspace{2 cm}
\displaystyle{+\beta(1-e^{-\bar\tau})\left\{ \max_{j\in\{1,...,N\}} \left|\frac{d{x}_j}{dt}(t)\right|+\left|\frac{d{x}_0}{dt}(t)\right|\right\} .}
\end{array}
$$
From Lemma \ref{lemma2} and Lemma \ref{lemma3} we have that
$$
\begin{array}{l}
\displaystyle{ D^+L(t)\leq \left[ -\frac \gamma 2 \phi(2R)+2\beta(1-e^{-\bar\tau})(1+\gamma)\right] d_0(t)}\\
\hspace{1.8 cm}
\displaystyle{+\left[R_\gamma-\beta e^{-\bar\tau}+\beta(1-e^{-\bar\tau})(1+2\gamma)\right]\sigma_\tau(t)}\\
\hspace{1.8 cm}
\displaystyle{ -\beta \int_{t-\bar\tau}^t e^{-(t-s)} \int_s^t \left( \max_{j\in\{1,...,N\}} \left|\frac{d{x}_j}{d\sigma}(\sigma)\right|+\left|\frac{d{x}_0}{d\sigma}(\sigma)\right|\right) d\sigma ds.}
\end{array}
$$
We want to show that there exists $\beta^*>0$ such that
\begin{equation}\label{scopo}
D^+L(t)\leq -\beta^* L(t)
\end{equation}
for any $t\geq \bar\tau$. To this aim, we need that
\begin{equation}\label{eq1}
-\frac \gamma 2 \phi(2R)+2\beta (1-e^{-\bar\tau}) (1+\gamma)<0,
\end{equation}
and
\begin{equation}\label{eq2}
R_\gamma-\beta e^{-\bar\tau}+\beta(1-e^{-\bar\tau})(1+2\gamma)\leq 0.
\end{equation}
From \eqref{eq1} we immediately have that $\beta$ must satisfy
$$
\beta < \frac{\gamma \phi(2R)}{4(1+\gamma)(1-e^{-\bar\tau})}.
$$
Moreover, \eqref{eq2} is satisfied for
$$
\beta \geq \frac{R_\gamma}{e^{-\bar\tau}-(1+2\gamma)(1-e^{-\bar\tau})},
$$
with the restriction on the time delay
$$
\bar\tau < \ln \left(1+\frac{1}{1+2\gamma}\right).
$$
For the existence of a constant $\beta$ satisfying both inequalities, we need
$$
\frac{R_\gamma}{e^{-\bar\tau}-(1+2\gamma)(1-e^{-\bar\tau})}<  \frac{\gamma \phi(2R)}{4(1+\gamma)(1-e^{-\bar\tau})}.
$$
This gives us a  further condition on the delay, namely
$$
\bar\tau< \ln\left( 1+\frac{\gamma\phi(2R)}{\gamma(1+2\gamma)\phi(2R)+4(1+\gamma)R_\gamma}\right),
$$
which clearly implies the previous restriction. Therefore, if the bound on the time delay $\bar\tau$ satisfies \eqref{cond_tau}, then inequality \eqref{scopo} is satisfied for a suitable positive constant $\beta^*.$ Hence, the theorem is proved.
\end{proof}

\section{Local controllability}\label{section_2}
In this section we will prove the existence of an admissible control which steers each agent to a given state $\overline{x}\in\RR^d$, which is the leader's aim. This shows the local controllability of system \eqref{main_model}, extending the local controllability result proved in \cite{WCB} for the undelayed model. First of all, we have the following lemma.
\begin{Lemma}\label{lemma4}
Let $\{x_i\}_{i\in\{0,1,...,N\}}$ be the solution to \eqref{main_model}-\eqref{initial_condition}, with $u(t)=0$ for any $t\geq 0$ and $x_{0,0}(s)=x_{0}$, for any $s\in[-\bar\tau,0]$. If the initial conditions satisfy for any $i\in\{1,...,N\}$
$$
\max_{s\in[-\bar\tau,0]} |x_{i}(s)-x_0|\leq \frac \delta 2,
$$
and the following assumption on the parameters
\begin{equation}\label{delta_ip}
\phi\left(\frac \delta 2\right)>\frac{1}{2\gamma},
\end{equation}
is satisfied, then there exist two constants $C,K>0$ such that
\begin{equation}\label{tesi_lemma4}
 |x_i(t)-x_0|\leq Ce^{-Kt},
\end{equation}
for any $i\in\{1,...,N\}$ and for any $t\geq 0$.
\end{Lemma}
\begin{proof}
First of all, we claim that if $\max_{s\in[-\bar\tau,0]} |x_{i}(s)-x_0|\leq \frac \delta 2$ for any $i\in\{1,...,N\}$, then
\begin{equation}\label{claim}
|x_{i}(t)-x_0|\leq \frac \delta 2,
\end{equation}
for any $t>0$. In order to prove this claim, suppose by contradiction that this is not true, namely there exists $\bar{t}>0$ such that for some $i\in\{1,...,N\}$,
$$
|x_i(\bar{t})-x_0|>\frac \delta 2.
$$
Then, by continuity there exists a time $t^*\in(0,\bar{t})$ such that $|x_i(t^*)-x_0|=\frac \delta 2$, $|x_j(t)-x_0|<\frac \delta 2$, for any $t<t^*$ and for any $j\in\{1,...,N\}$, and
$$
D^+|x_i(t^*)-x_0|\geq 0.
$$
However, we have that
$$
\begin{array}{l}
\displaystyle{\frac 1 2 D^+|x_i(t^*)-x_0|^2 \leq \frac 1 N \sum_{j\neq i} a_{ij}(t^*)\langle x_j(t^*-\tau(t^*))-x_i(t^*),x_i(t^*)-x_0\rangle }\\
\hspace{4 cm}
\displaystyle{-\gamma\phi(|x_0-x_i(t^*)|)|x_0-x_i(t^*)|^2}\\
\hspace{3.1 cm}
\displaystyle{\leq \frac 1 N \sum_{j\neq i}a_{ij}(t^*)|x_j(t^*-\tau(t^*))-x_0|\frac \delta 2 -\frac 1 N \sum_{j\neq i} a_{ij}(t^*) \left( \frac \delta 2 \right) ^2}\\
\hspace{4 cm}
\displaystyle{-\gamma \phi\left( \frac \delta 2\right) \left(\frac \delta 2 \right) ^2<0,}
\end{array}
$$
which is in contradiction with the previous assumption. Hence, the claim is true. In order to prove \eqref{tesi_lemma4}, we consider for any $t\geq 0$ the function
$$
F(t):=\max_{i\in\{1,...,N\}} |x_i(t)-x_0|^2.
$$
As in section \ref{section_consensus}, without loss of generality, we assume $F(t)=|x_1(t)-x_0|^2$. Then, for any $t\geq 0$ we have that
$$
\begin{array}{l}
\displaystyle{D^+F(t)=2\left\langle x_1(t)-x_0,\frac{d{x}_1}{dt}(t)\right\rangle }\\
\hspace{1.425 cm}
\displaystyle{ \leq\frac 2 N \sum_{j\neq 1} \langle x_j(t-\tau(t))-x_1(t),x_1(t)-x_0\rangle-2\gamma\phi(|x_0-x_1(t)|)F(t)}\\
\hspace{1.425 cm}
\displaystyle{ \leq \frac 2 N \sum_{j\neq 1} \langle x_j(t-\tau(t))-x_0,x_1(t)-x_0\rangle -\frac{2(N-1)}{N}F(t)-2\gamma\phi\left( \frac \delta 2 \right) F(t).}
\end{array}
$$
Using Cauchy-Schwarz and Young's inequalities yields
$$
\begin{array}{l}
\displaystyle{ D^+F(t)\leq \frac{N-1}{N} F(t-\tau(t))-2\gamma\phi\left( \frac \delta 2 \right) F(t)}\\
\hspace{1.425 cm}
\displaystyle{ \leq F(t-\tau(t))-2\gamma\phi\left( \frac \delta 2 \right) F(t).}
\end{array}
$$
Since hypothesis \eqref{delta_ip} holds, then by Halanay inequality (see e.g. \cite{Halanay, MG}) we get the existence of two positive constants $\tilde C, \tilde K>0$ such that for any $t\geq 0$,
$$
F(t)\leq \tilde C e^{-\tilde {K}t},
$$
which clearly implies \eqref{tesi_lemma4}.
\end{proof}
Before proving the main result, we need the following additional lemma.
\begin{Lemma}\label{lemma5}
Let $\{x_i\}_{i\in\{0,1,...,N\}}$ be the solution to \eqref{main_model}-\eqref{initial_condition}. Let $\xi\in\RR^d$. Then, there exists a control $u(\cdot):[0,+\infty)\to\RR^d$ such that $x_0$ reaches $\xi$ in finite time, namely there exists $t_0>0$ such that $x_0(t_0)=\xi$. Furthermore, if
$$
\tr:=\max_{i\in\{0,1,...,N\}}\max_{s\in[-\bar\tau,0]} |x_i(s)-\xi|,
$$
then,
\begin{equation}\label{passaggio}
\max_{i\in\{0,1,...,N\}} |x_i(t)-\xi|\leq \tr,
\end{equation}
for any $t\geq 0$.
\end{Lemma}
\begin{proof}
Let us take for any $t\geq 0$
\begin{equation}\label{control2}
u(t):= \begin{cases} M \frac{\xi-x_0(t)}{|\xi-x_0(t)|}, \ \text{if} \ x_0(t)\neq \xi,\\
0, \ \text{if} \ x_0(t)=\xi.
\end{cases}
\end{equation}
Then, $u$ is an admissible control. Moreover, for any $t\geq 0$, if $x_0(t)\neq \xi$,
%\begin{equation}\label{stima_leader2}
$$
\begin{array}{l}
\displaystyle{ \frac 1 2 D^+ |x_0(t)-\xi|^2=\frac{M}{|x_0(t)-\xi|}\langle x_0(t)-\xi,\xi-x_0(t)\rangle=-M|x_0(t)-\xi|.}
\end{array}
$$
%\end{equation}
Therefore,
$$
|x_0(t)-\xi|\leq |x_0(0)-\xi|-Mt, \quad t\in[0,t_0],
$$
where $t_0=\frac{|x_0(0)-\xi|}{M}$. We have then proved the first part of the statement.

In order to prove \eqref{passaggio}, first observe that, by construction, if $|x_0(0)-\xi|\leq\tilde{R}$, then
\begin{equation}\label{dis_contr}
|x_0(t)-\xi|\leq\tr,
\end{equation}
for any $t\geq0$. For the other agents, we use a similar argument as in the proof of Lemma \ref{lemma1}. Indeed, let $\tr_\epsilon:=\tr+\epsilon$. Set
$$
\mathcal{T}_\epsilon:=\left\{ t\in[0,+\infty) \ : \  \max_{i\in\{1,...,N\}} |x_i(s)-\xi|<\tr_\epsilon, \ \forall s\in [0,t)\right\}.
$$
By continuity, $\mathcal{T}_\epsilon\neq \emptyset$. Let us set $T_\epsilon:=\sup \mathcal{T}_\epsilon$. We suppose by contradiction that $T_\epsilon<+\infty$. Hence,
\begin{equation}\label{contr4}
\lim_{t\to T_\epsilon^ -} \max_{i\in\{1,...,N\}} |x_i(t)-\xi|=\tr_\epsilon.
\end{equation}
Moreover, from \eqref{main_model}, for any $t< T_\epsilon,$
$$
\begin{array}{l}
\displaystyle{\frac 1 2 D^+ |x_i(t)-\xi|^2= \frac 1 N \sum_{j\neq i} a_{ij}(t)\langle x_j(t-\tau(t))-x_i(t),x_i(t)-\xi\rangle}\\
\hspace{3 cm}
\displaystyle{+\gamma \phi(|x_0(t-\tau(t))-x_i(t)|)\langle x_0(t-\tau(t))-x_i(t),x_i(t)-\xi\rangle}\\
\hspace{2.8 cm}
\displaystyle{= \frac 1 N \sum_{j\neq i} a_{ij}(t)\left( \langle x_j(t-\tau(t))-\xi,x_i(t)-\xi\rangle-|x_i(t)-\xi|^2\right) }\\
\hspace{3 cm}
\displaystyle{+\gamma \phi(|x_0(t-\tau(t))-x_i(t)|)\left( \langle x_0(t-\tau(t))-\xi,x_i(t)-\xi\rangle -|x_i(t)-\xi|^2\right).}
\end{array}
$$
Using Cauchy-Schwarz inequality and the fact that
$$
|x_j(t-\tau(t))-\xi|\leq \tr_\epsilon,
$$
for any $j\in\{1,...,N\}$ and for any $t<T_\epsilon$, we get
$$
\begin{array}{l}
\displaystyle{ \frac 1 2 D^+ |x_i(t)-\xi|^2 \leq \frac 1 N \sum_{j\neq i} a_{ij}(t) |x_i(t)-\xi| \left( \tr_\epsilon-|x_i(t)-\xi|\right)}\\
\hspace{3 cm}
\displaystyle{+\gamma \phi(|x_0(t-\tau(t))-x_i(t)|)|x_i(t)-\xi|\left( \tr_\epsilon-|x_i(t)-\xi|\right).}
\end{array}
$$
Since $\tr_\epsilon-|x_i(t)-\xi|>0$, for $t\in\mathcal{T}_\epsilon$, we deduce
$$
\frac 1 2 D^+|x_i(t)-\xi|^2\leq (1+\gamma)|x_i(t)-\xi|(\tr_\epsilon -|x_i(t)-\xi|),
$$
which gives us
$$
D^+|x_i(t)-\xi|\leq (1+\gamma)\tr_\epsilon-(1+\gamma)|x_i(t)-\xi|.
$$
Then, using Gronwall estimate yields
\begin{equation}\label{stella2}
|x_i(t)-\xi|\leq e^{-(1+\gamma)t} (|x_i(0)-\xi|-\tr_\epsilon)+\tr_\epsilon <\tr_\epsilon, \quad \forall \ t<T_\epsilon.
\end{equation}
From \eqref{dis_contr} and \eqref{stella2}, we have
$$
\lim_{t\to T_\epsilon^-} \max_{i\in\{0,1,...,N\}} |x_i(t)-\xi| <\tr_\epsilon,
$$
which is in contradiction with \eqref{contr4}. By sending $\epsilon\to 0$, we obtain \eqref{passaggio}.
\end{proof}
We are now ready to prove the following theorem.
\begin{Theorem}\label{teorema2}
For any $\bar{x}\in\RR^d$, if the maximal time delay $\bar\tau$ satisfies \eqref{cond_tau} and \eqref{delta_ip} holds, then there exists an admissible control $u$ which steers all agents to $\bar{x}$, namely
\begin{equation}\label{tesi_teorema2}
\lim_{t\to +\infty} x_i(t)=\bar{x},
\end{equation}
for any $i\in\{0,1,...,N\}$.
\end{Theorem}
\begin{proof}
Let us fix $\bar{x}\in\RR^d$. In Section \ref{section_consensus} we have shown that there exists an admissible control $u$ so that consensus, which we denote by $x^*\in\RR^d$, occurs. We introduce a finite sequence $\{z_k\}_{k\in\{0,1,...,L\}}$, with $L\in\nat$, such that $z_0=x^*$, $z_L=\bar{x}$ and for any $k\in\{ 0,1,...,L-1\}$,
$$
|z_k-z_{k+1}|\leq \frac{\delta}{4}.
$$
Therefore, by taking $u$ as in \eqref{control}, there exists a positive time $t_0>0$ such that
$$
\max_{i\in\{0,1,...,N\}}\max_{t\in [t_0-\bar\tau,t_0]} |x_i(t)-x^*| \leq \frac{\delta}{4}.
$$
Now, for $t\geq t_0$, consider the control
$$
u(t)=M\frac{z_1-x_0(t)}{|z_1-x_0(t)|}.
$$
Hence, by Lemma \ref{lemma5}, there exists another time $t_{01}>t_0$ such that $x_0(t_{01})=z_1$ and
$$
|x_i(t)-z_1|\leq \frac{\delta}{2},
$$
for any $i\in\{1,...,N\}$ and $t\in[t_0,t_{01}]$. At this point, take $u(t)=0$ for $t\geq t_{01}$. Then, it's possible to apply Lemma \ref{lemma4}. Therefore, we have the existence of another time $t_1>0$ such that
$$
\max_{t\in[t_1-\bar\tau,t_1]}|x_i(t)-z_1|\leq \frac{\delta}{4},
$$
for any $i\in\{1,...,N\}$. For any $t>t_1$, take
$$
u(t)=M\frac{z_2-x_0(t)}{|z_2-x_0(t)|}.
$$
Then, we apply again Lemma \ref{lemma5}, in order to have another time $t_{11}>t_1$ such that $x_0(t_{11})=z_2$ and
$$
|x_i(t)-z_2|\leq \frac{\delta}{2},
$$
for any $i\in\{1,...,N\}$ and $t\in[t_1,t_{11}]$. Take once again $u(t)=0$ for any $t\geq t_{11}$. Then, from Lemma \ref{lemma4} there exists $t_2>0$ such that
$$
\max_{t\in[t_2-\bar\tau,t_2]}|x_i(t)-z_2|\leq \frac{\delta}{4},
$$
for any $i\in\{1,...,N\}$. Now, one can repeat the same argument so that $x_0$ reaches the values $z_k$ in finite times $t_k$ respectively for any $k\in\{1,...,L-1\}$, and for any $i\in\{1,...,N\}$,
$$
\max_{t\in[t_k-\bar\tau,t_k]}|x_i(t)-z_k|\leq \frac{\delta}{4}.
$$
Finally, there exists a time $\bar{t}>0$ such that $ x_0(\bar{t})=\bar{x}$. Take once again $u(t)=0$ for any $t\geq \bar{t}$. Then, from Lemma \ref{lemma4}
$$
\lim_{t\to +\infty} x_i(t)=\bar{x},
$$
for any $i\in\{1,...,N\}$. Therefore, we get the existence of an admissible control $u$ such that \eqref{tesi_teorema2} is satisfied.
\end{proof}

\section{Consensus for the HK model with distributed time delay}\label{sec_distr}
In this section we will prove the consensus behavior of solutions to the control system with distributed time delay \eqref{ds}. We consider the following control function:
\begin{equation}\label{control3}
u(t)=\frac{\gamma}{h(t)}\int_{t-\tau(t)}^t \beta(t-s) \alpha(t;s)\sum_{j=1}^N \phi(|x_j(s)-x_0(t)|)(x_j(s)-x_0(t))ds,
\end{equation}
where
$$
\alpha(t;s):=\frac 12 \min \left\{ \frac{\phi(|x_p(s)-x_0(t)|)}{N}, \frac{2M}{\gamma\sum_{j=1}^N |x_j(s)-x_0(t)|}\right\},
$$
for any $t\geq 0$ and $s\in[t-\tau(t),t]$ and $p\in\{1,...,N\}$ is such that
$$
|x_p(t)-x_0(t)|=d_0(t).$$
Once again, we will take, without loss of generality, $p=1$. Then, by definition of $\alpha$, $u$ is an admissible control for system \eqref{ds}. In order to prove the consensus behavior of solution, we give the following auxiliary lemmas.
\begin{Lemma}\label{lemma4.1}
Let $\{x_i\}_{i\in\{0,1,...,N\}}$ be the solution to \eqref{ds}-\eqref{initial_data_dd} with control $u(t)$ as in \eqref{control3} for any $t\geq 0$. Let $x_{i,0},x_{0,0}$ be continuous initial data on the time interval $[-\bar\tau,0]$ and define
$$
R:=\max_{i\in\{0,1,...,N\}} \max_{s\in[-\bar\tau,0]} |x_i(s)|.
$$
Then,
\begin{equation}\label{tesi_lemma4.1}
\max_{i\in\{0,1,...,N\}} |x_i(t)|\leq R,
\end{equation}
for any $t\geq 0$.
\end{Lemma}
\begin{proof}
Let us define, as in Lemma \ref{lemma1}, $R_\epsilon:=R+\epsilon$. Set
$$
\mathcal{M}_\epsilon :=\{ t\geq 0 \ : \ \max_{i\in\{0,1,...,N\}} |x_i(s)|<R_\epsilon, \ \forall s\in[0,t)\}.
$$
By continuity, $\mathcal{M}_\epsilon\neq\emptyset$. Then, let us define $T_\epsilon:=\sup \mathcal{M}_\epsilon$, and suppose by contradiction that it is finite. Then,
\begin{equation}\label{contr0}
\lim_{t\to T_\epsilon -} \max_{i\in\{0,1,...,N\}} |x_i(t)|=R_\epsilon.
\end{equation}
On the other hand, we have that
$$
\begin{array}{l}
\displaystyle{\frac 12 \deriv |x_0(t)|^2 \leq\frac{\gamma}{h(t)} \sum_{j=1}^N \int_{t-\tau(t)}^t \beta(t-s)\alpha(t;s)\phi(|x_j(s)-x_0(t)|)\left( \langle x_j(s),x_0(t)\rangle - |x_0(t)|^2\right) ds }\\
\hspace{2.2 cm}
\displaystyle{ \leq \frac{\gamma}{h(t)}\sum_{j=1}^N\int_{t-\tau(t)}^t \beta(t-s)\alpha(t;s) \phi(|x_j(s)-x_0(t)|) |x_0(t)|\left( R_\epsilon-|x_0(t)|\right) ds.}
\end{array}
$$
Since $|x_0(t)|<R_\epsilon$ for any $t<T_\epsilon$ and using the fact that $N\alpha(t;s)\leq 1$ for any $t\geq 0$ and $s\in[t-\tau(t),t]$, we get
$$
\begin{array}{l}
\displaystyle{\frac 12 \deriv |x_0(t)|^2 \leq \gamma |x_0(t)|(R_\epsilon-|x_0(t)|),}
\end{array}
$$
which clearly implies (by Gronwall estimate)
\begin{equation}\label{dis1}
|x_0(t)|\leq e^{-\gamma t } (|x_0(0)|-R_\epsilon) +R_\epsilon<R_\epsilon,
\end{equation}
for any $t<T_\epsilon$. Furthermore, for any $i\in\{1,...,N\}$ and $t<T_\epsilon$,
$$
\begin{array}{l}
\displaystyle{ \frac 12 \deriv|x_i(t)|^2 \leq \frac{1}{Nh(t)} \sum_{j\neq i} \int_{t-\tau(t)}^t \beta(t-s) a_{ij}(t;s) \left( \langle x_j(s),x_i(t)\rangle -|x_i(t)|^2\right) ds}\\
\hspace{2.4 cm}
\displaystyle{ +\frac{\gamma}{h(t)} \int_{t-\tau(t)}^t \beta(t-s)\phi(|x_0(s)-x_i(t)|) \left( \langle x_0(s),x_i(t)\rangle -|x_i(t)|^2\right) ds}\\
\hspace{2.1 cm}
\displaystyle{ \leq \frac{1}{Nh(t)} \sum_{j\neq i} \int_{t-\tau(t)}^t \beta(t-s) a_{ij}(t;s) |x_i(t)|\left( R_\epsilon -|x_i(t)|\right) ds}\\
\hspace{2.4 cm}
\displaystyle{ +\frac{\gamma}{h(t)} \int_{t-\tau(t)}^t \beta(t-s)\phi(|x_0(s)-x_i(t)|) |x_i(t)|\left( R_\epsilon -|x_i(t)|\right) ds,}\\
\end{array}
$$
where we have used once again the fact that $|x_0(t)|,|x_j(t)|<R_\epsilon $ for any $t<T_\epsilon$ and any $j\in\{1,...,N\}$. Then, we get
$$
\frac 12 \deriv |x_i(t)|^2\leq (\gamma+1) |x_i(t)|(R_\epsilon-|x_i(t)|),
$$
for any $t<T_\epsilon$. Hence, using again Gronwall estimate yields
\begin{equation}\label{dis2}
|x_i(t)|\leq e^{-(\gamma+1)t} (|x_i(0)|-R_\epsilon)+R_\epsilon<R_\epsilon.
\end{equation}
From inequalities \eqref{dis1}-\eqref{dis2}, we get
$$
\lim_{t\to T_\epsilon -} \max_{i\in\{0,1,...,N\}} |x_i(t)| <R_\epsilon,
$$
which is in contradiction with \eqref{contr0}. Then, $T_\epsilon=+\infty$. By sending $\epsilon\to 0$, we obtain \eqref{tesi_lemma4.1}.
\end{proof}
From Lemma \ref{lemma4.1} we obtain once again a bound on $\phi$ from above. Indeed, since $|x_0(t)|,|x_i(t)|\leq R$ for every $t\geq 0$ and for any $i\in\{1,...,N\}$, then
$$
\phi(|x_0(s)-x_i(t)|)\geq \phi(2R),
$$
for any $t\geq 0$ and $s\in[t-\tau(t),t]$. Also in this case, we could substitute the constant $R$ with the constant $R^*$ defined in Remark \ref{Rstar}.

The following lemma gives us an estimate on the Dini derivative of $d_0(t)$ for any $t\geq 0$, which is crucial in the proof of the consensus result.
\begin{Lemma}\label{lemma4.2}
Let $\{x_i\}_{i\in\{0,1,...,N\}}$ be the solution to \eqref{ds}-\eqref{initial_data_dd} with control $u(t)$ as in \eqref{control3} for any $t\geq 0$. If we define
$$
\lambda_\tau(t):=\frac{1}{h(t)} \int_{t-\bar\tau}^t \beta(t-s)\int_s^t \left[ \max_{j\in\{1,...,N\}} \left|\frac{d{x}_j}{d\sigma}(\sigma)\right|+|u(\sigma)|\right] d\sigma ds,
$$
for any $t\geq \bar\tau$, then
\begin{equation}\label{tesi_lemma4.2}
D^+d_0(t)\leq -\frac \gamma 2 \phi(2R) d_0(t)+ (\gamma LR+\gamma+1)\lambda_\tau(t),
\end{equation}
for any $t\geq \bar\tau$, where $L$ is the Lipschitz constant of $\phi(\cdot)$.
\end{Lemma}
\begin{proof} Assume, as above, $d_0(t)= \vert x_1(t)-x_0(t)\vert.$
From \eqref{ds} we immediately have
$$
\begin{array}{l}
\displaystyle{\frac 12 D^+ |x_1(t)-x_0(t)|^2\leq \frac{1}{Nh(t)} \sum_{j\neq 1 }\int_{t-\tau(t)}^t \beta (t-s) a_{1j}(t;s) \langle x_j(s)-x_1(t),x_1(t)-x_0(t)\rangle ds}\\
\hspace{2.4 cm}
\displaystyle{+\frac{\gamma}{h(t)} \int_{t-\tau(t)}^t \beta (t-s) \phi(|x_0(s)-x_1(t)|) \langle x_0(s)-x_1(t),x_1(t)-x_0(t)\rangle ds}\\
\hspace{2.4 cm}
\displaystyle{-\frac{\gamma}{h(t)}\sum_{j=1}^N \int_{t-\tau(t)}^t \beta(t-s)\alpha(t;s) \phi(|x_j(s)-x_0(t)|)\langle x_j(s)-x_0(t),x_1(t)-x_0(t)\rangle ds.}
\end{array}$$
By using similar argument as in Lemma \ref{lemma2}, we get
$$
\begin{array}{l}
\displaystyle{\frac 12 D^+ |x_1(t)-x_0(t)|^2\leq\frac{1}{Nh(t)} \sum_{j\neq 1 }\int_{t-\tau(t)}^t \beta (t-s) a_{1j}(t;s) \langle x_j(s)-x_j(t),x_1(t)-x_0(t)\rangle ds}\\
\hspace{2.4 cm}
\displaystyle{+\frac{1}{Nh(t)} \sum_{j\neq 1 }\int_{t-\tau(t)}^t \beta (t-s) a_{1j}(t;s) ds \langle x_j(t)-x_1(t),x_1(t)-x_0(t)\rangle}\\
\hspace{2.4 cm}
\displaystyle{ +\frac{\gamma}{h(t)} \int_{t-\tau(t)}^t \beta (t-s) \phi(|x_0(s)-x_1(t)|) \langle x_0(s)-x_0(t),x_1(t)-x_0(t)\rangle ds}\\
\hspace{2.4 cm}
\displaystyle{-\frac{\gamma}{h(t)} \int_{t-\tau(t)}^t \beta (t-s) \phi(|x_0(s)-x_1(t)|) ds | x_0(t)-x_1(t)|^2 }\\
\hspace{2.5 cm}
\displaystyle{-\frac{\gamma}{h(t)}\sum_{j=1}^N \int_{t-\tau(t)}^t \beta(t-s)\alpha(t;s) \phi(|x_j(s)-x_0(t)|)\langle x_j(s)-x_j(t),x_1(t)-x_0(t)\rangle ds}\\
\hspace{2.4 cm}
\displaystyle{-\frac{\gamma}{h(t)}\sum_{j=1}^N \int_{t-\tau(t)}^t \beta(t-s)\alpha(t;s) \phi(|x_j(s)-x_0(t)|)ds\langle x_j(t)-x_0(t),x_1(t)-x_0(t)\rangle .}
\end{array}
$$
Since $a_{1j}(t;s)\leq 1 $ for any $ j\in\{1,...,N\}$ and $t\geq 0$, then
$$
\begin{array}{l}
\displaystyle{ \frac 12 D^+ |x_1(t)-x_0(t)|^2\leq (\gamma+1) \lambda_\tau(t)-\frac{\gamma}{h(t)}\int_{t-\tau(t)}^t \beta(t-s)\phi(|x_0(s)-x_1(t)|)ds |x_1(t)-x_0(t)|^2}\\

\displaystyle{+\frac{1}{h(t)} \sum_{j=1}^N \int_{t-\tau(t)}^t \beta(t-s) \left( \left| \frac 1N a_{1j}(t;s)-\gamma\alpha(t;s)\phi(|x_j(s)-x_0(t)|)\right| -\frac 1 N a_{1j}(t;s)\right) ds |x_1(t)-x_0(t)|^2.}\\
\end{array}
$$
Define for any $j\in\{1,...,N\}$ the following quantity:
$$
b_j:=\left| \frac 1N a_{1j}(t;s)-\gamma\alpha(t;s)\phi(|x_j(s)-x_0(t)|)\right| -\frac 1 N a_{1j}(t;s).
$$
If $j$ is such that
$$
\frac 1N a_{1j}(t;s)-\gamma\alpha(t;s)\phi(|x_j(s)-x_0(t)|)\geq 0,
$$
then $b_j\leq 0$. Otherwise, we obtain
$$
b_j\leq \gamma \alpha(t;s)\phi(|x_j(s)-x_0(t)|.
$$
Therefore, we can deduce
$$
\begin{array}{l}
\displaystyle{ \frac 12 D^+|x_1(t)-x_0(t)|^2 \leq (\gamma+1)\lambda_\tau(t)-\frac{\gamma}{h(t)}\int_{t-\tau(t)}^t \beta(t-s)\phi(|x_0(s)-x_1(t)|)ds |x_1(t)-x_0(t)|^2}\\
\hspace{3.5 cm}
\displaystyle{+\frac{\gamma}{h(t)}\int_{t-\tau(t)}^t N\beta(t-s)\alpha(t;s)ds |x_1(t)-x_0(t)|^2.}
\end{array}
$$
Now, we observe that $N\alpha(t;s)\leq \frac 1 2 \phi(|x_1(s)-x_0(t)|)$, for any $t\geq 0$. Therefore,
$$
\begin{array}{l}
\displaystyle{\frac 12 D^+|x_1(t)-x_0(t)|^2 \leq (\gamma+1)\lambda_\tau(t) -\frac \gamma 2 \phi(2R) |x_1(t)-x_0(t)|^2}\\
\hspace{2 cm}
\displaystyle{+\frac{\gamma}{2h(t)}\int_{t-\tau(t)}^t \beta(t-s)\left|\phi(|x_0(s)-x_1(t)|)-\phi(|x_1(s)-x_0(t)|)\right| ds \,|x_1(t)-x_0(t)|^2.}
\end{array}
$$
Since $\phi(\cdot)$ is Lipschitz and $|x_1(t)-x_0(t)|\leq 2R$ for any $t\geq 0$, we get
$$
\begin{array}{l}
\displaystyle{ \frac 12 D^+|x_1(t)-x_0(t)|^2 \leq -\frac \gamma 2 \phi(2R)|x_1(t)-x_0(t)|^2+(\gamma RL+\gamma+1)\lambda_\tau(t)|x_1(t)-x_0(t)|,}
\end{array}
$$
which immediately implies \eqref{tesi_lemma4.2}. Hence, the lemma is proved.
\end{proof}
The following last technical lemma of this section holds.
\begin{Lemma}\label{lemma4.3}
Let $\{ x_i\}_{i\in\{0,1,...,N\}}$ be the solution to \eqref{ds}-\eqref{initial_data_dd} with control $u(t)$ as in \eqref{control3} for any $t\geq 0$. Then,
\begin{equation}\label{tesi1_lemma4.3}
\max_{i\in\{1,...,N\}} \left|\frac{d{x}_i}{dt}(t)\right|\leq (2+\gamma)d_0(t)+(1+\gamma)\lambda_\tau(t),
\end{equation}
for any $t\geq \bar\tau$. Moreover,
\begin{equation}\label{tesi2_lemma4.3}
\left|\frac{d{x}_0}{dt}(t)\right|\leq \gamma d_0(t)+\gamma\lambda_\tau(t),
\end{equation}
for any $t\geq \bar\tau$.
\end{Lemma}
\begin{proof}
From \eqref{ds} we immediately have for any $i\in\{1,...,N\}$,
$$
\begin{array}{l}
\displaystyle{\left|\frac{d{x}_i}{dt}(t)\right|\leq \frac{1}{Nh(t)}\sum_{j\neq i} \int_{t-\tau(t)}^t \beta (t-s) |x_j(s)-x_i(t)|ds+\frac{\gamma}{h(t)} \int_{t-\tau(t)}^t \beta(t-s) |x_0(s)-x_i(t)|ds}\\
\hspace{1.45 cm}
\displaystyle{\leq \frac{1}{Nh(t)}\sum_{j\ne i}\int_{t-\tau(t)}^t \beta(t-s) |x_j(s)-x_j(t)|ds + \frac 1 N\sum_{j\ne i}|x_j(t)-x_i(t)|}\\
\hspace{3 cm}
\displaystyle{+\frac{\gamma}{h(t)}\int_{t-\tau(t)}^t \beta(t-s) |x_0(s)-x_0(t)| ds+\gamma |x_0(t)-x_i(t)|.}
\end{array}
$$
Then, from the definition of $\lambda_\tau$ and using the fact that $|x_0(t)-x_i(t)|\leq d_0(t)$ for any $t\geq 0$, we get
$$
\left|\frac{d{x}_i}{dt}(t)\right|\leq (\gamma+1)\lambda_\tau(t)+(\gamma+2)d_0(t).
$$
By taking the maximum function, we obtain \eqref{tesi1_lemma4.3}. Moreover,
$$
\begin{array}{l}
\displaystyle{\left|\frac{d{x}_0}{dt}(t)\right|\leq \frac{\gamma}{h(t)} \int_{t-\tau(t)}^t \beta(t-s)\alpha(t;s)\sum_{j=1}^N |x_j(s)-x_j(t)|ds+\gamma d_0(t)}\\
\hspace{1.1 cm}
\displaystyle{ \leq \gamma \lambda_\tau(t)+\gamma d_0(t),}
\end{array}
$$
obtaining \eqref{tesi2_lemma4.3}.
\end{proof}
Now, we are ready to prove the main theorem of this section.
\begin{Theorem}\label{teorema3}
Let $\{x_i\}_{i\in\{0,1,...,N\}}$ be the solution to \eqref{ds}-\eqref{initial_data_dd} with control $u(t)$ as in \eqref{control3} for any $t\geq 0$. Define
$$
B:=\int_0^{\bar\tau} \beta(s)ds.
$$
If the maximal time delay satisfies the following inequality
\begin{equation}\label{delay_restr}
\bar\tau < \ln \left( 1+\frac{\gamma\phi(2R)}{4R_\gamma B (1+\gamma)+\gamma\phi(2R) B (1+2\gamma)}\right),
\end{equation}
with $R_\gamma:=\gamma LR+\gamma+1$, then the solution tends towards consensus, namely there exist two positive constants $C,K>0$ such that
\begin{equation}\label{tesi_teo3}
|x_i(t)-x_0(t)|\leq Ce^{-Kt},
\end{equation}
for any $i\in\{1,...,N\}$ and $t\geq 0$.
\end{Theorem}
\begin{proof}
Let us define the following Lyapunov functional
$$
\mathcal{L}(t):=d_0(t)+\mu \int_0^{\bar\tau} \beta(s)\int_{t-s}^t e^{-(t-\sigma)} \int_\sigma^t \left( \max_{j\in\{1,...,N\}} \left|\frac{d{x}_j}{d\rho}(\rho)\right|+\left|\frac{d{x}_0}{d\rho}(\rho)\right|\right)d\rho d\sigma ds,\quad t\ge\bar\tau,
$$
where $\mu>0$ is a constant which will be determined later. By differentiating $\mathcal{L}$ in time, we get
$$
\begin{array}{l}
\displaystyle{D^+ \mathcal{L}(t) = D^+ d_0(t)-\mu \int_0^{\bar\tau}\beta (s) e^{-s}\int_{t-s}^t \left( \max_{j\in\{1,...,N\}} \left|\frac{d{x}_j}{d\rho}(\rho)\right|+\left|\frac{d{x}_0}{d\rho}(\rho)\right|\right) d\rho ds }\\
\hspace{1.7 cm}
\displaystyle{-\mu \int_0^{\bar\tau} \beta(s)\int_{t-s}^t e^{-(t-\sigma)} \int_{\sigma}^t \left( \max_{j\in\{1,...,N\}} \left|\frac{d{x}_j}{d\rho}(\rho)\right|+\left|\frac{d{x}_0}{d\rho}(\rho)\right|\right)d\rho d\sigma ds}\\
\hspace{1.7 cm}
\displaystyle{+\mu\left(  \max_{j\in\{1,...,N\}} \left|\frac{d{x}_j}{dt}(t)\right|+\left|\frac{d{x}_0}{dt}(t)\right|  \right)\int_0^{\bar\tau} \beta (s) \left( 1-e^{-s}\right) ds.}
\end{array}
$$
From Lemma \ref{lemma4.2} and Lemma \ref{lemma4.3} and using the fact that $e^{-s}\geq e^{-\bar\tau}$ for any $s\in[0,\bar\tau]$, we obtain
$$
\begin{array}{l}
\displaystyle{D^+\mathcal{L}(t)\leq \left[ -\frac \gamma 2 \phi(2R) +2\mu B \left( 1-e^{-\bar\tau}\right) (\gamma+1)\right] d_0(t)}\\
\hspace{1.6 cm}
\displaystyle{+ \left[ R_\gamma -\mu e^{-\bar\tau}+\mu B \left( 1 -e^{-\bar\tau}\right) (2\gamma+1)\right] \lambda_\tau(t)}\\
\hspace{1.6 cm}
\displaystyle{-\mu \int_0^{\bar\tau} \beta(s)\int_{t-s}^t e^{-(t-\sigma)} \int_\sigma^t \left( \max_{j\in\{1,...,N\}} \left|\frac{d{x}_j}{d\rho}(\rho)\right|+\left|\frac{d{x}_0}{d\rho}(\rho)\right|\right)d\rho d\sigma ds,}
\end{array}
$$
We want to show that there exists $K>0$ such that
\begin{equation}\label{aim2}
D^+ \mathcal{L}(t)\leq -K\mathcal{L}(t),
\end{equation}
for any $t\geq \bar\tau$. To do so, we look for $\mu>0$ such that
$$
-\frac \gamma 2 \phi(2R) +2\mu B \left( 1-e^{-\bar\tau}\right) (\gamma+1) <0,
$$
and
$$
R_\gamma -\mu e^{-\bar\tau}+\mu B \left( 1 -e^{-\bar\tau}\right) (2\gamma+1)\leq 0.
$$
The first inequality is satisfied for
$$
\mu< \frac{\gamma \phi(2R)}{4B (1+\gamma)(1-e^{-\bar\tau})},
$$
while the second inequality holds true for
$$
\mu \geq \frac{R_\gamma}{e^{-\bar\tau} -B(1+2\gamma)(1-e^{-\bar\tau})},
$$
under the restriction
$$
\bar\tau < \ln \left( 1+ \frac{1}{B(1+2\gamma)} \right).
$$
In order to have the existence of $\mu $ satisfying both inequalities, we need
$$
\frac{R_\gamma}{e^{-\bar\tau} -B (1+2\gamma)(1-e^{-\bar\tau})} <  \frac{\gamma \phi(2R)}{4B (1+\gamma)(1-e^{-\bar\tau})},
$$
This is satisfied for
$$
\bar\tau < \ln \left( 1+\frac{\gamma\phi(2R)}{4R_\gamma B (1+\gamma)+\gamma\phi(2R)B (1+2\gamma)}\right),
$$
which clearly implies the previous restriction. Therefore, if $\bar\tau$ satisfies \eqref{delay_restr}, inequality \eqref{aim2} is satisfied for some $K>0$. Hence, \eqref{tesi_teo3} immediately follows.
\end{proof}
\section{Local controllability for the HK model with distributed time delay}\label{sec5}
In this section we extend the theory presented in Section \ref{section_2} to the system with distributed time delay \eqref{ds}. We will prove that for any $\bar{x}\in\RR^d$, there exists a control $u$ thanks to which all agents reach the leader's desired opinion, namely
\begin{equation}\label{iph}
\lim_{t\to +\infty} x_i(t)=\bar x,
\end{equation}
for any $i\in\{0,1,...,N\}$. To do so, we need the following lemma.
\begin{Lemma}\label{lemma5.1}
Let $\{x_i\}_{i\in\{0,1,...,N\}}$ be the solution to \eqref{ds}-\eqref{initial_data_dd} with $u(t)=0$ for any $t\geq 0$. Let us suppose that
\begin{equation}\label{N}
\max_{s\in[-\bar\tau,0]} |x_i(s)-x_0|\leq \frac \delta 2,
\end{equation}
for any $i\in\{1,...,N\}$. Moreover, assume \eqref{delta_ip}. Then, there exist two constants $C,K>0$ such that
\begin{equation}\label{tesi_lemma5.1}
|x_i(t)-x_0|\leq Ce^{-Kt},
\end{equation}
for any $i\in\{1,...,N\}$ and $t\geq 0$.
\end{Lemma}
\begin{proof}
First we claim that if \eqref{N} is satisfied, then for any $t\geq 0$ and $i\in\{1,...,N\}$,
$$
|x_i(t)-x_0|\leq \frac \delta 2.
$$
Arguing by contradiction, let us suppose that this is not true. Hence, there exists a time $\bar t >0$ such that for some $i\in\{1,...,N\}$, $|x_i(\bar t)-x_0|>\frac \delta 2$. Then, by continuity, there exists another time $t_0\in (0,\bar t)$ such that
$$
|x_i(t_0)-x_0|=\frac \delta 2, \qquad D^+ |x_i(t_0)-x_0|\geq 0,
$$
and for any $t<t_0$ and $j\in\{1,...,N\}$, $|x_j(t)-x_0|<\frac \delta 2$. On the other hand, we have that
$$
\begin{array}{l}
\displaystyle{\frac 12 D^+|x_i(t_0)-x_0|^2=\left\langle x_i(t_0)-x_0,\frac{d{x}_i}{dt}(t_0)\right\rangle}\\
\hspace{3.15 cm}
\displaystyle{=\frac{1}{Nh(t_0)} \sum_{j\neq i} \int_{t_0-\tau(t_0)}^{t_0} \beta (t_0-s) \langle x_j(s)-x_i(t_0),x_i(t_0)-x_0\rangle ds}\\
\hspace{3.15 cm}
\displaystyle{-\gamma \phi(|x_0-x_i(t_0)|)|x_i(t_0)-x_0|^2.}
\end{array}
$$
We observe that $\langle x_j(s)-x_i(t_0),x_i(t_0)-x_0\rangle\leq 0$ for any $s\in[t_0-\tau(t_0),t_0]$. Indeed, Cauchy-Schwarz inequality yields
$$
\langle x_j(s)-x_i(t_0),x_i(t_0)-x_0\rangle\leq |x_j(s)-x_0||x_i(t_0)-x_0|-|x_i(t_0)-x_0|^2 \leq \frac \delta 2 |x_j(s)-x_0|-\left(\frac \delta 2\right) ^2 \leq 0.
$$
Therefore,
$$
\frac 12 D^+|x_i(t_0)-x_0|^2 \leq -\gamma \phi(|x_0-x_i(t_0)|)|x_i(t_0)-x_0|^2 <0,
$$
which is in contradiction with the previous assumption. Hence, the claim is proved. Now, in order to prove \eqref{tesi_lemma5.1}, we define
$$
F(t):=\max_{i\in\{1,...,N\}} |x_i(t)-x_0|^2.
$$
Without loss of generality, we can take $F(t)=|x_1(t)-x_0|^2$. Then,
$$
\begin{array}{l}
\displaystyle{ D^+ F(t) \leq \frac{2}{Nh(t)} \sum_{j\neq 1} \int_{t-\tau(t)}^t \beta(t-s) \langle x_j(s)-x_1(t),x_1(t)-x_0\rangle ds -2\gamma\phi\left( \frac \delta 2 \right) F(t)}\\
\hspace{1.45 cm}
\displaystyle{\leq \frac{2}{Nh(t)} \sum_{j\neq 1} \int_{t-\tau(t)}^t \langle x_j(s)-x_0,x_1(t)-x_0\rangle ds}\\
\hspace{4 cm}
\displaystyle{-\frac {2(N-1)}{N} F(t)-2\gamma \phi\left( \frac \delta 2 \right) F(t).}
\end{array}
$$
By Cauchy-Schwarz inequality and Young inequality, we get
$$
\begin{array}{l}
\displaystyle{ D^+F(t)\leq \frac{1}{h(t)}\int_{t-\tau(t)}^t \beta (t-s)F(s) ds-2\gamma \phi\left( \frac \delta 2 \right) F(t)}\\
\hspace{1.45 cm}
\displaystyle{\leq \sup_{s\in[t-\tau(t),t]} F(s) -2\gamma \phi\left( \frac \delta 2 \right) F(t).}
\end{array}
$$
From Halanay-type inequality (see \cite{MG}), since \eqref{delta_ip} is assumed, we get \eqref{tesi_lemma4.1}. Hence, the lemma is proved.
\end{proof}
Moreover, analogously to the pointwise time delay case, the following result can be proved.
\begin{Lemma}\label{lemma5.2}
Let $\{x_i\}_{i\in\{0,1,...,N\}}$ be the solution to \eqref{ds}-\eqref{initial_data_dd}. Let $\xi\in\RR^d$. Then, there exists a control $u(\cdot):[0,+\infty)\to\RR^d$ such that $x_0$ reaches $\xi$ in finite time, namely there exists $t_0>0$ such that $x_0(t_0)=\xi$. Furthermore, if
$$
\tr:=\max_{i\in\{0,1,...,N\}}\max_{s\in[-\bar\tau,0]} |x_i(s)-\xi|,
$$
then,
\begin{equation}\label{passaggio1}
\max_{i\in\{0,1,...,N\}} |x_i(t)-\xi|\leq \tr,
\end{equation}
for any $t\geq 0$.
\end{Lemma}
Finally, following the same argument as in Theorem \ref{teorema2}, one can prove the  local controllability result for system \eqref{ds}.
\begin{Theorem}
Let $\{x_i\}_{i\in\{0,1,...,N\}}$ be the solution to \eqref{ds}-\eqref{initial_data_dd}. If $\bar\tau$ satisfies inequality \eqref{delay_restr} and \eqref{delta_ip} holds true, then for every $\bar{x} \in \RR^d$ there exists a control $u(\cdot):[0,+\infty)\to\RR^d$ such that
$$
\lim_{t\to +\infty} x_i(t)=\bar{x},
$$
for any $i\in\{0,1,...,N\}$. Hence, we have local controllability for system \eqref{ds}.
\end{Theorem}

\section{Numerical simulations}\label{numeric}
In this section we give some numerical tests illustrating  the results of Section \ref{section_consensus} and Section \ref{section_2} for particular choices of $a(\cdot)$, $\phi(\cdot)$ and different constant-in-time time delays. We will use the Matlab function \emph{dde23} to solve numerically system \eqref{main_model}. We take $N=50$ agents with constant initial data given by
$$
x_i(t)=(-1)^i \frac{i}{50},
$$
for any $i\in\{1,...,50\}$ and $t\leq 0$. Moreover, we choose $\delta=1$, $r=2$, and $a_{ij}$ of the form
$$
a_{ij}(t)=
\begin{cases}
1, \quad \text{if} \ |x_j(t-\tau)-x_i(t)|\leq 1,\\
2-|x_j(t-\tau)-x_i(t)|\quad \text{if} \ |x_j(t-\tau)-x_i(t)|\in (1,2), \\
0 \quad \text{if} \ |x_j(t-\tau)-x_i(t)|>2,
\end{cases}
$$
for any $t\geq 0$ and $i,j\in\{1,...,50\}$. Finally, we take $\gamma=1$ and $\phi(\cdot)$ as the classical Cucker-Smale influence function, namely of the form
$$
\phi(s)= \frac{1}{(1+s^2)^{3/2}},
$$
for any $s\geq 0$. In Figure \ref{figura1} we consider the control function as in \eqref{control} and we look for solutions according to different time delays. In particular, we find different solutions corresponding to $\tau=0.5, 1, 10$ and $25$. We notice the oscillatory behavior of solutions with $\tau=10$ and $\tau=25$ once the system has reached consensus. Moreover, note that the bigger the delays are, the more the leader's opinion has to change in order to have consensus. As expected, this will take, for higher delays, a lot of time in order to bring all agents to a prefixed state $\bar x$.
\begin{figure}[t]
\centering
\includegraphics[scale=0.25]{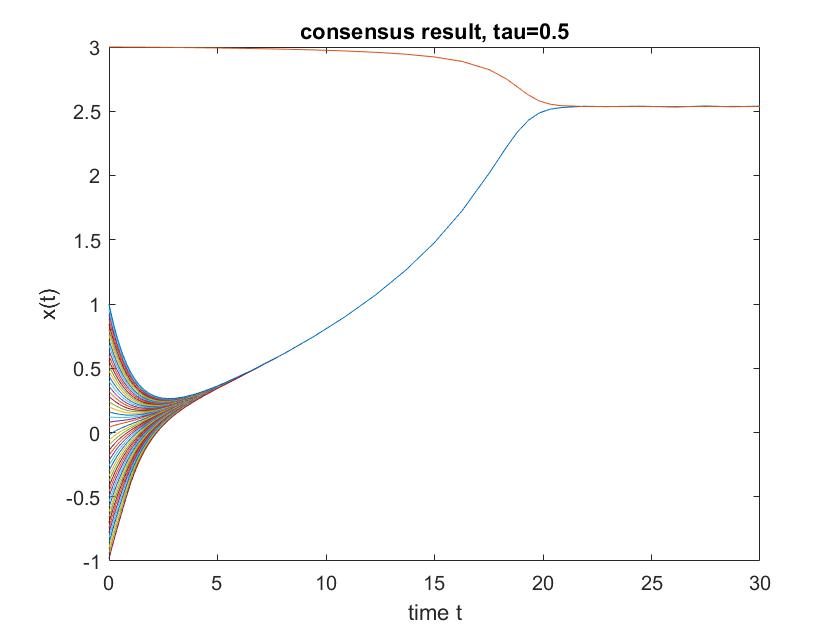}
\includegraphics[scale=0.25]{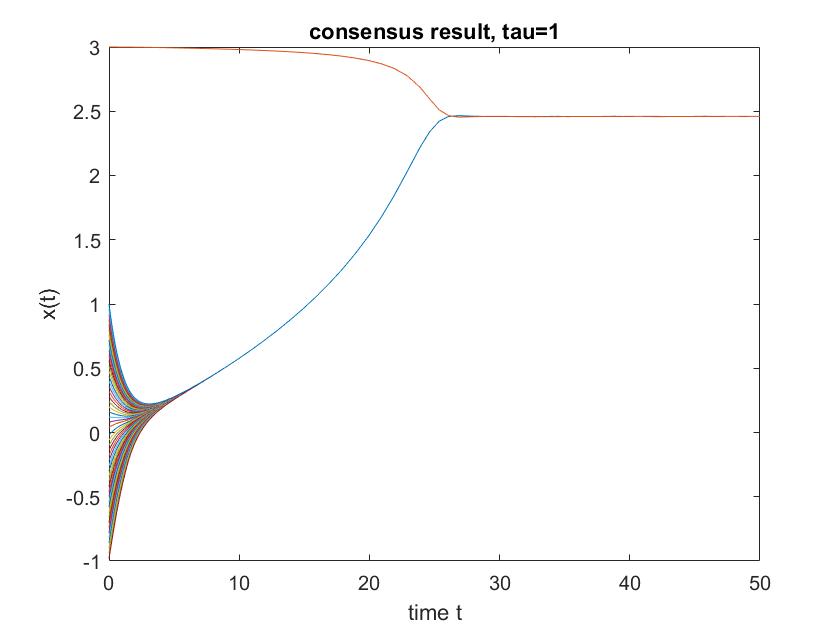}\\
\includegraphics[scale=0.25]{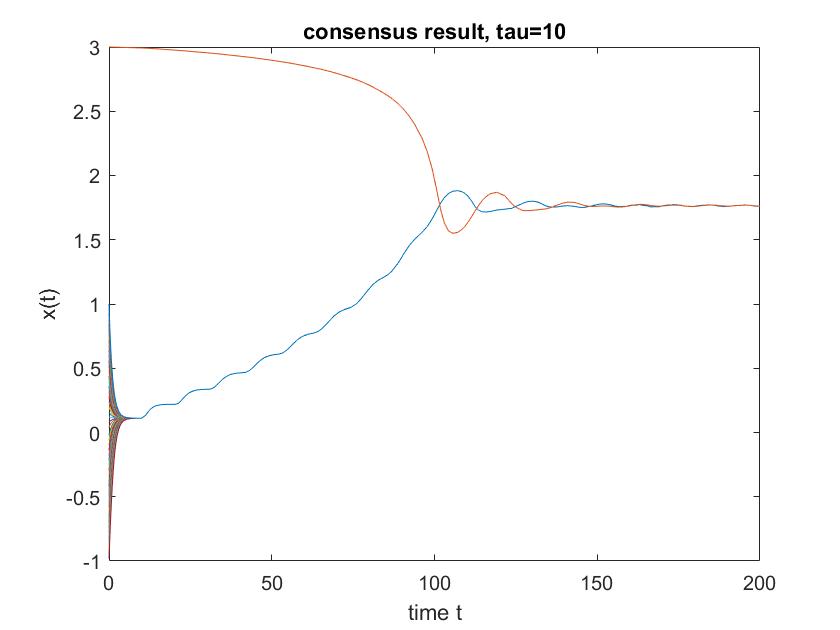}
\includegraphics[scale=0.25]{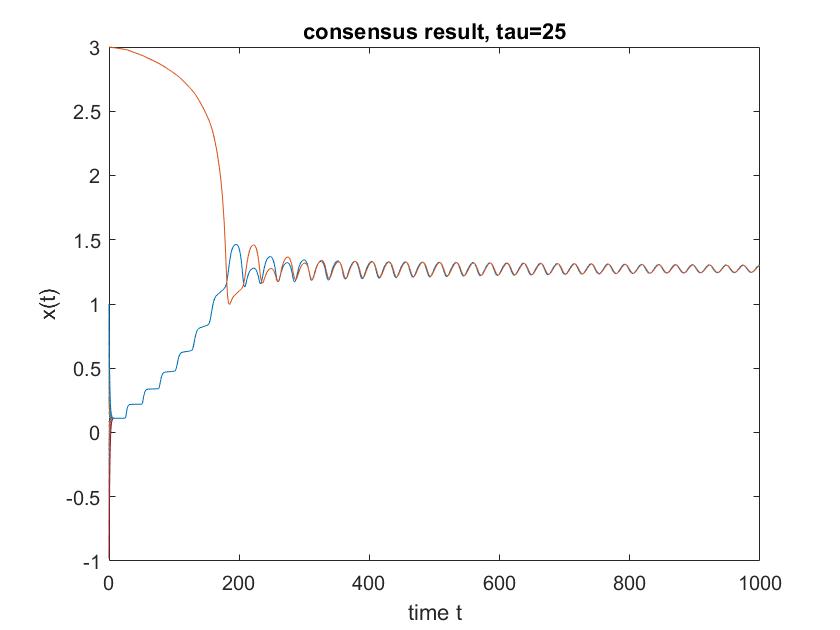}
\caption{Time evolution of solutions with different values of time delays; $\tau=0.5$ (top left), $\tau=1$ (top right), $\tau=10$ (bottom left), $\tau=25$ (bottom right), corresponding to control function $u$ as in \eqref{control}.}
\label{figura1}
\end{figure}

In Figure \ref{controllability_image} we consider $\tau(t)\equiv 1$ for any $t\geq 0$. Since condition \eqref{delta_ip} is satisfied, we can apply the proof of Theorem \ref{teorema2} to system \eqref{main_model} with $N,a(\cdot),\gamma$ and $\phi(\cdot)$ as above. We fix $\bar{x}=4$ and, according to Theorem \ref{teorema2}, we show numerically that there esixts a control function $u$ such that the leader's opinion reaches in finite time the consensus state $\bar{x}$ and for any $i\in\{1,...,50\}$,
$$
\lim_{t\to +\infty} x_i(t)=\bar x.
$$
Hence, we obtain the local controllability result.

\begin{figure}[t]
\centering
\includegraphics[scale=0.5]{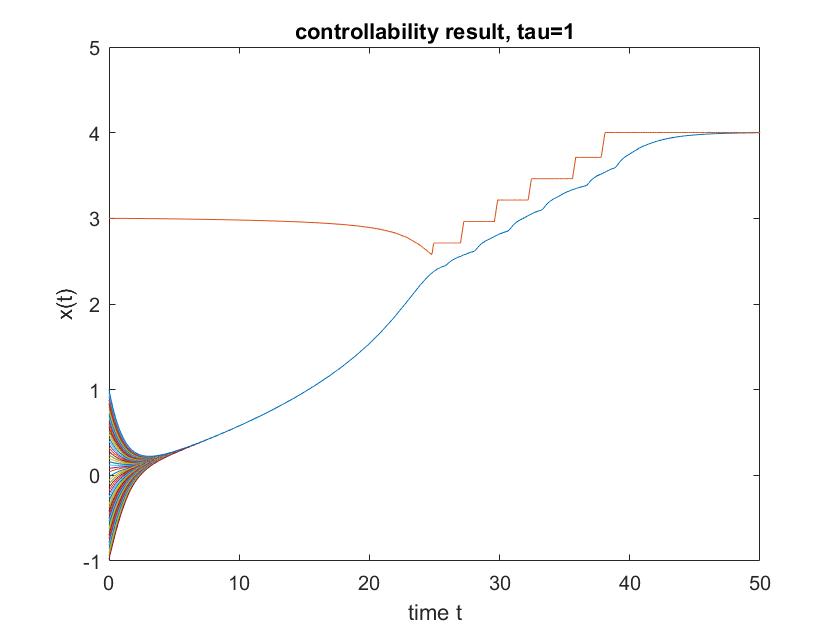}
\caption{Local controllability result applied to system \eqref{main_model} with $\tau=1$.}
\label{controllability_image}
\end{figure}

\end{document}